\documentclass{amsart}

\usepackage{mathptmx,amsmath, amssymb, amscd, paralist,tabularx,supertabular,amsmath, verbatim, amsthm, amssymb, mathrsfs, manfnt,latexsym,amscd,graphicx,hyperref,color}

\author{Benjamin Linowitz}
\address{Department of Mathematics\\Oberlin College\\Oberlin, OH 44074}
\email{benjamin.linowitz@oberlin.edu}

\author{D. B. McReynolds}
\address{Department of Mathematics\\Purdue University\\West Lafayette, IN 47907}
\email{dmcreyno@purdue.edu}

\author{Paul Pollack}
\address{Department of Mathematics\\University of Georgia\\Athens, GA 30602}
\email{pollack@uga.edu}

\author{Lola Thompson}
\address{Department of Mathematics\\Oberlin College\\Oberlin, OH 44074}
\email{lola.thompson@oberlin.edu}

\title{Systoles of Arithmetic Hyperbolic \\ Surfaces and 3--manifolds}

\usepackage{fancyhdr}

\fancypagestyle{plain}

\DeclareMathAlphabet{\curly}{U}{rsfs}{m}{n}

\DeclareMathOperator{\Comm}{Comm}

\DeclareMathOperator{\covol}{covol}

\DeclareMathOperator{\disc}{disc}

\DeclareMathOperator{\Ram}{Ram}

\DeclareMathOperator{\Gal}{Gal}

\DeclareMathOperator{\Mat}{M}

\DeclareMathOperator{\N}{N}

\DeclareMathOperator{\Norm}{Norm}

\DeclareMathOperator{\PSL}{PSL}

\DeclareMathOperator{\Reg}{Reg}

\DeclareMathOperator{\Sys}{Sys}

\DeclareMathOperator{\Vol}{Vol}

\newtheorem{thm}{Theorem}[section]
\newtheorem{cor}[thm]{Corollary}

\newtheorem{prop}[thm]{Proposition}
\newtheorem{lem}[thm]{Lemma}
\theoremstyle{definition}

\newtheorem*{rmk}{Remark}

\theoremstyle{remark}

\newtheorem{proposition}{Proposition}[section]
\def\1{\mathbf{1}}
\def\disc{\mathrm{disc}}

\theoremstyle{remark}

\theoremstyle{plain}
\newtheorem{theorem}[proposition]{Theorem}

\newtheorem{lemma}[proposition]{Lemma}

\def\C{\mathbf{C}}
\def\Q{\mathbf{Q}}

\def\pp{\mathfrak{p}}

\def\Pp{\curly{P}}

\def\R{\mathbf{R}}

\def\N{\mathbf{N}}
\def\Q{\mathbf{Q}}

\def\1{\mathbf{1}}
\def\disc{\mathrm{disc}}
\def\Gal{\mathrm{Gal}}

\newcommand{\frakp}{\mathfrak{p}}

\newcommand{\calO}{\mathcal{O}}

\newcommand{\innp}[1]{\left< #1 \right>}
\newcommand{\abs}[1]{\left\vert#1\right\vert}
\newcommand{\set}[1]{\left\{#1\right\}}

\newcommand{\iny}{\infty}

\newcommand{\leg}[2]{\genfrac{(}{)}{}{}{#1}{#2}}
\makeatletter
\def\moverlay{\mathpalette\mov@rlay}
\def\mov@rlay#1#2{\leavevmode\vtop{%
   \baselineskip\z@skip \lineskiplimit-\maxdimen
   \ialign{\hfil$\m@th#1##$\hfil\cr#2\crcr}}}
\newcommand{\charfusion}[3][\mathord]{
    #1{\ifx#1\mathop\vphantom{#2}\fi
        \mathpalette\mov@rlay{#2\cr#3}
      }
    \ifx#1\mathop\expandafter\displaylimits\fi}
\makeatother

\makeatletter
\let\@@pmod\pmod
\DeclareRobustCommand{\pmod}{\@ifstar\@pmods\@@pmod}
\def\@pmods#1{\mkern4mu({\operator@font mod}\mkern 6mu#1)}
\makeatother

\begin{document}

\begin{abstract}
Our main result is that for all sufficiently large $x_0>0$, the set of commensurability classes of arithmetic hyperbolic 2-- or 3--orbifolds with fixed invariant trace field $k$ and systole bounded below by $x_0$ has density one within the set of all commensurability classes of arithmetic hyperbolic 2-- or 3--orbifolds with invariant trace field $k$. The proof relies upon bounds for the absolute logarithmic Weil height of algebraic integers due to Silverman, Brindza and Hajdu, as well as precise estimates for the number of rational quaternion algebras not admitting embeddings of any quadratic field having small discriminant. When the trace field is $\Q$, using work of Granville and Soundararajan, we establish a stronger result that allows our constant lower bound $x_0$ to grow with the area. As an application, we establish a systolic bound for arithmetic hyperbolic surfaces that is related to prior work of Buser--Sarnak and Katz--Schaps--Vishne. Finally, we establish an analogous density result for commensurability classes of arithmetic hyperbolic 3--orbifolds with small area totally geodesic $2$--orbifolds.
\end{abstract}

%---------------------------------------------------------------------------------
%------------------- Front data -----------------------------------------------
\maketitle

%---------------------------------------------------------------------------------
%%% Introduction
%---------------------------------------------------------------------------------
\section{Introduction}

Given a closed, orientable surface $\Sigma_g$ of genus $g \geq 2$, the moduli space of hyperbolic metrics on $\Sigma_g$ is denoted by $\mathcal{M}_g$. This $(3g-3)$ complex dimensional moduli space is a central object of interest in several fields. The present article is concerned with the subset of arithmetic hyperbolic points. It follows from work of Borel \cite{borel-commensurability} that the set of arithmetic hyperbolic structures comprise a finite set in $\mathcal{M}_g$. These very special hyperbolic metrics naturally arise in connection to algebraic and geometric extremal problems; for instance, Hurwitz surfaces that achieve the maximal possible order isometry group are always arithmetic.

Associated to a hyperbolic metric is a discrete, faithful representation $\rho_{hol}\colon \pi_1(\Sigma_g) \to \PSL(2,\R)$ with image that we denote by $\Gamma$. By seminal work of Margulis \cite{Mar}, a hyperbolic metric is arithmetic if and only if $[\Comm(\Gamma):\Gamma] = \infty$, where the commensurator $\Comm(\Gamma)$ is given by
\[ \Comm(\Gamma) = \set{\eta \in \PSL(2,\R)~:~[\Gamma:\Gamma_\eta],~[\Gamma^\eta:\Gamma_\eta] < \infty} \]
and $\Gamma^\eta = \eta^{-1}\Gamma\eta$, $\Gamma_\eta = \Gamma \cap \Gamma^\eta$. There are several (conjectural) characterizations of arithmeticity based on algebraic information about $\Gamma$ and geometric information about the metric itself (see Cooper--Long--Reid \cite{CLR,Reid}, Geninska--Leuzinger \cite{GL}, Lafont--McReynolds \cite{LM}, Luo--Sarnak \cite{LuoS}, and Schmutz \cite{Schmutz}). These characterizations function through symmetries and it is not entirely clear what makes arithmetic hyperbolic surfaces special geometrically through the above lenses.

One well-known conjecture regarding the special geometric nature of arithmetic hyperbolic surfaces is the Short Geodesic Conjecture. For a hyperbolic surface $M \in \mathcal{M}_g$, we denote the systole of $M$ by $\Sys(M)$, and recall that this is the length of the shortest closed geodesic on $M$. The Short Geodesic Conjecture asserts that there exists a constant $C>0$, independent of genus, such that if $M$ is an arithmetic hyperbolic surface, then $\Sys(M) \geq C$. There is an analogous conjectural uniform lower bound for the systole of arithmetic hyperbolic 3--orbifolds. That Short Geodesic Conjecture for arithmetic hyperbolic 3--orbifolds is slightly stronger than the Salem Conjecture (see \cite[\S 12.3]{MR}) that asserts a uniform lower bound on the set of Mahler measures of the Salem polynomials.

In order to state our main results, we require some additional notation and terminology. In \S \ref{section:prelims} we will define the volume $V_{\mathscr C}$ of a commensurability class $\mathscr{C}$ of arithmetic orbifolds. This volume is defined in terms of the volume of a distinguished representative of the class which arises in a natural way from a maximal order in the quaternion algebra associated to $\mathscr C$ and allows us to count the number of commensurability classes with bounded volume. Define $N_k(V)$ to be the number of commensurability classes $\mathcal{C}$ of arithmetic hyperbolic $2$--orbifolds (respectively, $3$--orbifolds) with invariant trace field $k$ and $V_\mathcal{C} <V$. Given $x_0 > 0$, define $N_k(V;x_0)$ to be the number of classes $\mathscr{C}$ with invariant trace field $k$, $V_\mathcal{C} < V$, and which have a representative $M \in \mathscr{C}$ satisfying $\Sys(M) < x_0$. We will make use of standard asymptotic notation from analytic number theory throughout. We will use interchangeably the Vinogradov symbol, $f\ll g$, and the Landau Big-Oh notation, $f = O(g)$, to indicate that there is a constant $C>0$ such that $\abs{f}\leq C\abs{g}$. Moreover, we will write $f\asymp g$ to indicate that $f\ll g$ and $g\ll f$. Lastly, we write $f = o(g)$ if $\lim_{x \rightarrow \infty} \frac{f(x)}{g(x)} = 0$ and $f \sim g$ if $\lim_{x \rightarrow \infty} \frac{f(x)}{g(x)} = 1.$ Any subscripts on these symbols will indicate dependence of the implied constants.

\begin{thm}\label{theorem:shortgeodesiccommensurabilityclasses}
For all sufficiently large $x_0$, we have $N_k(V;x_0) \asymp V/\log(V)^{\frac{1}{2}}$ as $V$ tends to infinity, where the implied constants depend only on $k$ and $x_0$.
\end{thm}

By establishing that $N_k(V)\asymp V$, we deduce the following density result from Theorem \ref{theorem:shortgeodesiccommensurabilityclasses}.

\begin{cor}\label{T:SGC-Density}
For all sufficiently large $x_0$, and for every totally real number field $k$ (respectively, number field with exactly one complex place), we have $\lim_{V \to \iny} \frac{N_k(V;x_0)}{N_k(V)} = 0$.
\end{cor}

It is straightforward to see that there is a uniform lower bound for the systoles of arithmetic hyperbolic 2--orbifolds with fixed invariant trace field $k$. As a result, for small values of $x_0$ we will have $N_k(V;x_0) = 0$ for all $V$, in which case the statement of Corollary \ref{T:SGC-Density} is trivially satisfied. However, for $x_0$ sufficiently large, $N_k(V;x_0)$ is unbounded as $V$ tends to infinity. Our main result says that regardless of how large we fix our notion for ``short'' with regard to the systole, the density of commensurability classes that have a representative with a short geodesic is always zero.

If we restrict to the class of arithmetic hyperbolic surfaces arising from quaternion algebras over $\Q$, we can allow our notion of short to grow with the area of the surface while still maintaining density one. Namely, with density one the systole has order of magnitude at least $\log\log(V_\mathscr{C})$.

\begin{thm}\label{T:SystoleInequality}
Within the set of all commensurability classes of arithmetic hyperbolic surfaces with invariant trace field $\mathbf{Q}$ there is, for all $\epsilon>0$, a density one subset of classes $\mathscr{C}$ such that
\[\Sys(\mathbf{H}^2/\Gamma)>(\frac{1}{8}-\epsilon)\log\log\left(\frac{24}{\pi}V_\mathscr{C}\right)\]
holds for all $\Gamma \in \mathscr{C}$.
\end{thm}

If $\Gamma$ is a maximal arithmetic lattice with invariant trace field $\Q$ which has minimal co-area in its commensurability class, one can deduce from Theorem \ref{T:SystoleInequality} that with density one we have
\begin{equation}\label{E:MaxSystEst}
\Sys(\mathbf{H}^2/\Gamma)>(\frac{1}{8}-\epsilon)\log\log\left(\frac{24}{\pi}V \right),
\end{equation}
where $V$ is the co-area of $\Gamma$. As the systole is non-decreasing in covers, we see for those commensurability classes, we have that as a uniform lower bound. In particular, we can compare (\ref{E:MaxSystEst}) with the prior systolic estimates of Buser--Sarnak \cite{BS} and Katz--Schaps--Vishne \cite{KSV}. In \cite{BS}, Buser and Sarnak proved that if $\Gamma$ is a cocompact arithmetic Fuchsian group defined over $\textbf{Q}$ then there is a constant $c=c(\Gamma)$ such that the systole of the congruence cover $\Gamma[I]$ satisfies
\begin{equation}\label{E:BS}
\Sys(\textbf{H}^2/\Gamma[I])>\frac{4}{3}\log\left(g(\textbf{H}^2/\Gamma[I])\right)-c,
\end{equation}
where $g(\cdot)$ denotes genus. This result was subsequently extended to arbitrary cocompact arithmetic Fuchsian groups in \cite{KSV}. Both \cite{BS} and \cite{KSV} made extensive use of careful trace estimates, whereas we use counting arguments that take advantage of the basic tools of multiplicative number theory. One component of the argument is a bound for a negative moment of $L(1,\chi)$, as $\chi$ ranges over quadratic Dirichlet characters; we extract this bound from the detailed study made by Granville--Soundararajan \cite{GS} of the distribution of these $L$--values. Theorem \ref{T:SystoleInequality} provides a density one lower bound with order of magnitude $\log\log$, which is not as good as (\ref{E:BS}) which has order of magnitude $\log$. On the other hand, \eqref{E:MaxSystEst} holds without a depth requirement, while \eqref{E:BS} becomes non-trivial only once the level is sufficiently big. Important here is that our method provides non-trivial bounds on the systole growth of a density one subset of commensurability classes of maximal arithmetic hyperbolic surfaces. That yields a depth-free bound on a density one set of classes.

\textbf{Remark.}
It is our use of the work of Granville--Soundararajan \cite{GS} that forces us to restrict to arithmetic surfaces with invariant trace field $\Q$.

One may view totally geodesic surfaces as an analogue of geodesics in hyperbolic $3$--orbifolds. With this motivation, our final result is an analogous density result for small area totally geodesic $2$--orbifolds in commensurability classes of arithmetic hyperbolic 3--orbifolds.

\begin{thm}\label{T:SSC-Density}
Let $V>0$ and $k_1,\dots,k_r$ be the invariant trace fields of the arithmetic hyperbolic $2$--orbifolds with area at most $V$. The set of commensurability classes of arithmetic hyperbolic $3$--orbifolds having a representative containing a totally geodesic $2$--orbifold with area at most $V$ has density zero within the set of all commensurability classes of arithmetic hyperbolic $3$--orbifolds that have invariant trace field given as a quadratic extension of some $k_i$.
\end{thm}

Obtaining uniform lower bounds on the area of the smallest totally geodesic $2$--orbifold is trivial since the area of \emph{any} finite type hyperbolic $2$--orbifold is uniformly bounded from below. However, a $2$--orbifold can arise as a totally geodesic $2$--orbifold in infinitely many incommensurable arithmetic hyperbolic $3$--orbifolds and so the above density result is non-trivial. 

The aforementioned density results are established using counting results that are of independent interest. In our prior work \cite{LMPT}, the main input from analytic number theory came in the guise of Tauberian theorems for Dirichlet series. Such results give a convenient method for translating information about singular points into asymptotic estimates. The main novelty in this paper is the use of mean value estimates for multiplicative functions that are valid  uniformly, instead of merely asymptotically. As with the counting results from \cite{LMPT}, these methods potentially have a much broader range of applications to other algebraic and geometric counting problems, and subsequently broader geometric applications. Indeed, this paper serves as an illustration of these applications.

%---------------------------------------------------------------------------------
%%% The construction of arithmetic hyperbolic manifolds
%---------------------------------------------------------------------------------
\section{Preliminaries}\label{section:prelims}

\paragraph{\textbf{Notation.}}

Throughout this paper $k$ will denote a number field of signature $(r_1(k),r_2(k))$. In practice $k$ will be either totally real or else contain a unique complex place. The ring of integers of $k$ will be denoted by $\calO_k$. Given an ideal $I$ of $\calO_k$, we will denote by $\abs{I}$ its norm. The set of prime ideals of $\mathcal O_k$ will be denoted $\Pp_k$. The degree of $k$ will be denoted by $n_k$, the discriminant by $\Delta_k$, the associated Dedekind zeta function by $\zeta_k(s)$ and the regulator by $\Reg_k$. If $L/k$ is a finite extension then we will denote by $\Delta_{L/k}$ the relative discriminant.

Let $k$ be a number field and $B$ be a quaternion algebra over $k$. The set of primes of $k$ which ramify in $B$ will be denoted by $\Ram(B)$. The subset of $\Ram(B)$ consisting of the finite (respectively infinite) primes of $k$ which ramify in $B$ will be denoted by $\Ram_f(B)$ (respectively $\Ram_\infty(B)$). We define the discriminant $\disc(B)$ of $B$ to be the product of all primes in $\Ram(B)$. Finally, $\disc_f(B),\disc_\infty(B)$ will denote the product of all the primes in $\Ram_f(B), \Ram_\infty(B)$.

Let $\textbf{H}^2$, $\textbf{H}^3$ denote real hyperbolic $2$--, $3$--space. For a lattice $\Gamma <\PSL(2,\bf{R})$, $\PSL(2,\bf{C})$, let $M = \textbf{H}^2/\Gamma$, $\textbf{H}^3/\Gamma$ denote the associated finite volume hyperbolic $2$-- or $3$--orbifold. The orbifold $M$ is a manifold precisely when $\Gamma$ is torsion-free (i.e., has no non-trivial elements of finite order). We will refer to any non-trivial element of $\Gamma$ of finite order as a \textbf{torsion element}. 

\paragraph{\textbf{Arithmetic Manifolds.}}

In this section we describe the construction of arithmetic lattices in $\PSL(2,\bf{R})$ and $\PSL(2,\bf{C})$. For a more detailed exposition we refer the reader to Maclachlan--Reid \cite{MR}.

Let $\Gamma_1,\Gamma_2$ be subgroups of $\PSL(2,\bf{C})$. We say that $\Gamma_1$ and $\Gamma_2$ are \textbf{directly commensurable} if $\Gamma_1\cap \Gamma_2$ has finite index in both $\Gamma_1$ and $\Gamma_2$. We say that $\Gamma_1$ and $\Gamma_2$ are \textbf{commensurable in the wide sense} if $\Gamma_1$ is directly commensurable with a $\PSL(2,\C)$--conjugate of $\Gamma_2$. Note that $\Gamma_1, \Gamma_2$ are commensurable in the wide sense if and only if the associated hyperbolic orbifolds $M_1,M_2$ have a common finite cover.

We begin by reviewing the construction of arithmetic Fuchsian groups. Let $k$ be a totally real field and $B$ a quaternion algebra over $k$ which is unramified at a unique real place $v$ of $k$. We therefore have an identification $B_v=B\otimes_k k_v\cong \Mat(2,\bf{R})$. Let $\calO$ be a maximal order of $B$ and $\calO^1$ the multiplicative group consisting of those elements of $\calO$ having reduced norm one. We denote by $\Gamma_{\calO}^1$ the image in $\PSL(2,\bf{R})$ of $\calO^1$. The group $\Gamma_{\calO}^1$ is a discrete subgroup of $\PSL(2,\bf{R})$ having finite covolume. A subgroup $\Gamma$ of $\PSL(2,\bf{R})$ is an \textbf{arithmetic Fuchsian group} if it is commensurable in the wide sense with a group of the form $\Gamma_{\calO}^1$ for some totally real field $k$, quaternion algebra $B$ over $k$ and maximal order $\calO$ of $B$. We will denote by $\mathscr C(k,B)$ the set of all discrete subgroups of $\PSL(2,\bf{R})$ commensurable with $\Gamma_{\calO}^1$.

The construction of arithmetic Kleinian groups is very similar. Let $k$ be a number field with a unique complex place $v$ and $B$ a quaternion algebra over $k$ which is ramified at all real places. Let $\calO$ be a maximal order of $B$ and $\Gamma_{\calO}^1$ the image in $\PSL(2,\bf{C})$ of $\calO^1$ under the identification $B_v=B\otimes_k k_v\cong \Mat(2,\bf{C})$. A subgroup $\Gamma$ of $\PSL(2,\bf{C})$ is an \textbf{arithmetic Kleinian group} if it is commensurable in the wide sense with a group of the form $\Gamma_{\calO}^1$ for some number field $k$ having a unique complex place, quaternion algebra $B$ over $k$ ramified at real primes and maximal order $\calO$ of $B$.

Given two arithmetic lattices $\Gamma_1,\Gamma_2$ arising from $(k_i,B_i)$, we know that $\Gamma_1$ and $\Gamma_2$ will be commensurable in the wide sense precisely when $k_1\cong k_2$ and $B_1\cong B_2$ \cite[Thm 8.4.1]{MR}. We will make use of this fact many times throughout the remainder of this paper.

Throughout this paper we will be interested in counting the number of commensurability classes of arithmetic hyperbolic surfaces or $3$--manifolds with a specified property. We will count these commensurability classes as follows. Let $\mathscr C(k,B)$ be a commensurability class of arithmetic hyperbolic surfaces or $3$--manifolds. We define the \textbf{volume of $\mathscr C(k,B)$} to be $V_{\mathscr C(k,B)}:=\covol(\Gamma_{\mathcal O}^1)$ where $\mathcal O$ is a maximal order in $B$. A result of Borel \cite{borel-commensurability} (see also \cite[Ch 11.1]{MR}) shows that
\begin{equation*}
V_{\mathscr{C}(k,B)}=\covol(\Gamma_{\calO}^1)=\frac{8\pi \abs{\Delta_k}^{\frac{3}{2}}\zeta_k(2)}{(4\pi^2)^{n_k}}\prod_{\frakp\mid\disc_f(B)}\left({\abs{\frakp}}-1\right)
\end{equation*}
when $\mathscr C(k,B)$ is a commensurability class of arithmetic hyperbolic surfaces and that
\begin{equation*}
V_{\mathscr{C}(k,B)}=\covol(\Gamma_{\calO}^1)=\frac{ \abs{\Delta_k}^{\frac{3}{2}}\zeta_k(2)}{(4\pi^2)^{n_k-1}}\prod_{\frakp\mid\disc_f(B)}\left({\abs{\frakp}}-1\right).
\end{equation*}
when $\mathscr C(k,B)$ is a commensurability class of arithmetic hyperbolic $3$--manifolds. Note that this definition does not depend on the choice of maximal order \cite[p. 336]{MR}. It is with respect to this notion of volume that all of our counting results for commensurability classes of arithmetic manifolds will be based. We note that this definition is slightly different than existing notions in the literature, where the volume of the commensurability class $\mathscr C(k,B)$ would be defined as the minimal volume achieved by a representative of $\mathscr C(k,B)$.

Setting $\Sys(\mathcal{C}(k,B)) = \inf_{M \in \mathcal{C}(k,B)}\Sys(M)$, there are infinitely many orbifolds in the commensurability class that realize $\Sys(\mathcal{C}(k,B))$. We will prove that $\Sys(\mathcal{C}(k,B))$ can always be realized by a manifold.

\begin{lemma}\label{L:OrbMan1Sys}
Given an arithmetic hyperbolic $2$-- or $3$--orbifold $M$ and a closed geodesic $c$ on $M$, there exists a finite manifold cover $M' \to M$ such that $c$ lifts to $M'$. In particular, there exists a finite manifold cover $M' \to M$ such that $\Sys(M) = \Sys(M')$.
\end{lemma}

\begin{proof}
By definition of arithmeticity, there is a pair $(k,B)$  such that $\Gamma = \pi_1^{\textrm{orb}}(M)$ is commensurable with $P\mathcal{O}^1$ where $\mathcal{O}$ is a maximal order in $B$. Associated to $c$ is a $\Gamma$--conjugacy class of hyperbolic elements $[\gamma] \subset \Gamma$ and it suffices to find a torsion-free, finite index subgroup $\Gamma_0 < \Gamma$ with $\gamma \in \Gamma_0$. Now, every torsion element $\eta \in \Gamma$ must have eigenvalues that are roots of unity that are contained in extensions of $k$ of uniformly bounded degree. Consequently, there is a bound on the order of the torsion elements that depends only on the degree of $k$, and we let $T$ be the least common multiple of these orders. If $\lambda_\gamma$ is an eigenvalue of $\gamma$ and $L=k(\lambda_\gamma)$, then $L/k$ is a quadratic extension, $\lambda_\gamma \in \mathcal{O}_L^1$, and $\Gamma < \PSL(2,\mathcal{O}_L)$ (up to squares). By \cite[Thm 2.3]{Ham}, for any sufficiently large $m \in \N$, there exists a prime ideal $\mathfrak{p}$ in $\mathcal{O}_L$ such that the image of $\lambda_\gamma$ in $\mathcal{O}_L/\mathfrak{p}$ has multiplicative order $m$. If $r_\mathfrak{p}\colon \Gamma \to \PSL(2,\mathcal{O}_L/\mathfrak{p})$ is given by reducing the matrix coefficients modulo $\mathfrak{p}$, then $r_\mathfrak{p}(\gamma)$ has order $m$ when $m$ is odd. For any odd $m$ that is relatively prime with $T$, if $m$ is sufficiently large, the characteristic $p$ of $\mathcal{O}_L/\mathfrak{p}$ will also be relatively prime with $T$. Selecting such a $m$ and reducing modulo a sufficiently large power of $\mathfrak{p}$, we can ensure for any torsion element $\eta \in \Gamma$ that the order of $r_{\mathfrak{p}^\ell}(\eta)$ is equal to the order of $\eta$. To see this assertion, observe that there are only finitely many possibilities for the characteristic polynomials of the torsion elements $\eta \in \Gamma$ and that the characteristic polynomial of $r_{\mathfrak{p}^\ell}(\eta)$ equals the characteristic polynomial of $\eta$ modulo $\mathfrak{p}^\ell$. Since every torsion element $\eta \in \Gamma$ has characteristic polynomial different from $(t-1)^2$ and there are only finitely possibilities for these characteristic polynomials, there exists $\ell_0 \in \N$ such that the characteristic polynomial modulo $\mathfrak{p}^{\ell_0}$ is not $(t-1)^2$ for every torsion element in $\Gamma$. For any  $\ell \geq \ell_0$, we know that $r_{\mathfrak{p}^\ell}(\eta) \ne 1$ for every torsion element and so the order of $r_{\mathfrak{p}^\ell}(\eta)$ will equal the order of $\eta$. Returning to the proof, as the order of $r_{\mathfrak{p}^\ell}(\gamma)$ is $mp^{\ell'}$ for some $\ell' \geq 0$ and $mp^{\ell'}$ is relatively prime to $T$,  we see that $r_{\mathfrak{p}^\ell}(\eta) \notin \innp{r_{\mathfrak{p}^\ell}(\gamma)}$ for every torsion element $\eta \in \Gamma$. Hence, $r_{\mathfrak{p}^\ell}^{-1}(\innp{r_{\mathfrak{p}^\ell}(\gamma)})$ is a finite index, torsion-free subgroup of $\Gamma$ that contains $\gamma$.
\end{proof}

\begin{rmk}
There are infinitely many manifolds in $\mathcal{C}(k,B)$ that realize $\Sys(\mathcal{C}(k,B))$ since the covers $M_\ell$ associated to $r_{\mathfrak{p}^\ell}^{-1}(\innp{r_{\mathfrak{p}^\ell}(\gamma)})$ are manifolds with $\Sys(M) = \Sys(M_\ell)$ for all sufficiently large $\ell$. Moreover, Lemma \ref{L:OrbMan1Sys} holds for any complete, finite volume hyperbolic $2$-- or $3$--orbifold. In the case of a hyperbolic 2--orbifold, note that one can deform the holonomy representation associated to the hyperbolic structure to obtain a new hyperbolic structure such that the invariant trace field is algebraic. 
\end{rmk}

%---------------------------------------------------------------------------------
%%%Arithmetic manifolds with short geodesics
%---------------------------------------------------------------------------------
\section{Absolute logarithmic Weil heights and lengths of closed geodesics}\label{Sec:SGC}

%---------------------------------------------------------------------------------
\subsection{Bounds for absolute logarithmic Weil heights}

In this section we count the number of commensurability classes of arithmetic hyperbolic $2$-- and $3$--orbifolds which have a fixed invariant trace field $k$ and possess a representative with a closed geodesic of length less than $x_0$. Our proof will make use of several important facts about absolute logarithmic Weil heights of algebraic integers, hence we begin by defining the relevant terms.

Let $L$ be a number field, $\pp\in \Pp_L$ a prime ideal and $\abs{\cdot}_\pp$ the associated valuation, normalized so that for each $\alpha\in L$, we have $\prod_{\pp\mid\infty} \abs{\alpha}_\pp=\abs{\Norm_{L/\Q}(\alpha)}$ and $\prod_{\pp\in \Pp_L}\abs{\alpha}_\pp=1$. We define the \textbf{logarithmic height of $\alpha$ relative to $L$} to be $h_L(\alpha)=\sum_{\pp\in \Pp_L} \log(\max\{1,\abs{\alpha}_\pp\})$. The \textbf{absolute logarithmic Weil height of $\alpha$} is $h(\alpha)=[L:\Q]^{-1}h_L(\alpha)$ and is independent of the field $L$. We will make repeated use of the fact that \textbf{height of $\alpha$ relative to $\Q(\alpha)$} is the logarithm of the Mahler measure of the minimal polynomial of $\alpha$. The following height bounds will play an important role in the proof of this section's main result.

\begin{theorem}[Silverman \cite{silverman}]\label{theorem:silvermanbound}
Let $L/k$ be a quadratic extension of number fields with norm of relative discriminant $\abs{\Delta_{L/k}}$ and $\alpha$ be a primitive element for the extension. Then the absolute logarithmic Weil height $h(\alpha)$ of $\alpha$ satisfies $$ h(\alpha)\geq \frac{-(r_1(k)+r_2(k))\log(2)}{2n_k}+\frac{1}{4n_k}\log\abs{\Delta_{L/k}}.$$
\end{theorem}

\begin{theorem}[Brindza \cite{brindza}, Hajdu \cite{hajdu}]\label{theorem:brindzabound}
Let $L$ be a number field of degree $n_L\geq 2$ with unit group rank $r_L$ and regulator $\Reg_L$. If $L$ is not an imaginary quadratic field, then there exists a system of fundamental units $\{ u_1,\dots,u_{r_L}\}$ of $L$ such that $h(u_i)\leq 6^{n_L}n_L^{5n_L}\Reg_L$ holds for all $1\leq i \leq r_L$.
\end{theorem}

In order to proceed we now translate the height bounds of Theorems \ref{theorem:silvermanbound} and \ref{theorem:brindzabound} into facts about the lengths of closed geodesics on certain arithmetic orbifolds.

\begin{prop}\label{prop:silvermanapplication}
Let $k$ be a totally real number field (respectively number field with a unique complex place) of degree $n_k$ and $B$ be a quaternion algebra over $k$ which is unramified at precisely one real place of $k$ (respectively ramified at all real places of $k$). If $\Sys(\mathcal{C}(k,B)) \leq x_0$, then there exists a quadratic extension $L/k$ which embeds into $B$ and satisfies $\abs{\Delta_{L/k}}<e^{2(n_k+x_0)}$.
\end{prop}

\begin{proof}
We give a proof in the case that $k$ has a unique complex place. The proof in the totally real case is similar and will be left to the reader. Suppose therefore that $\Gamma$ is an arithmetic Kleinian group in the commensurability class defined by $(k,B)$ such that the hyperbolic $3$--orbifold $\mathbf{H}^3/\Gamma$ contains a closed geodesic of length $\ell(\gamma)<x_0$, where $\gamma\in\Gamma$ is the associated hyperbolic element. It is known \cite[Ch 8]{MR} that in this case the subgroup $\Gamma^{(2)}$ of $\Gamma$ generated by squares is derived from a quaternion algebra in the sense that $\Gamma^{(2)}$ is contained in a group of the form $\Gamma_{\mathcal O}^1$ for some maximal order $\mathcal O$ of $B$. Note that $\Gamma^{(2)}$ contains the element $\gamma^2$, hence the quotient orbifold contains a closed geodesic of length $2\ell(\gamma)$. Denote by $\lambda_{\gamma^2}$ the unique eigenvalue of $\gamma^2$ with $|\lambda_{\gamma^2}|>1$. The length of the closed geodesic associated to $\gamma^2$ is equal to twice the logarithm of the Mahler measure of the minimal polynomial of $\gamma^2$ \cite[Lemma 12.3.3]{MR}, which is equal to the height of $\gamma^2$ relative to $\Q(\lambda_{\gamma^2})$. In particular, we see that $[\Q(\lambda_{\gamma^2}):\Q]\in \{n_k,2n_k\}$ (see \cite[Lemma 2.3]{chinburg-geodesics}), and so the absolute logarithmic Weil height $h(\gamma^2)$ of $\gamma^2$ satisfies $x_0>2n_kh(\gamma^2)$. As the field $L=k(\lambda_{\gamma^2})$ is a quadratic extension of $k$ which embeds into $B$ (see \cite[Ch 12]{MR}), the proposition follows from Theorem \ref{theorem:silvermanbound}.
\end{proof}

\begin{prop}\label{prop:brindzaapplication}
Let $k$ be a totally real number field (respectively number field with a unique complex place) of degree $n_k$ and absolute value of discriminant $\abs{\Delta_k}$ and $B$ be a quaternion algebra over $k$ which is unramified at precisely one real place of $k$ (respectively ramified at all real places of $k$). If $B$ admits an embedding of a quadratic extension $L/k$ which satisfies
\[ \abs{\Delta_{L/k}}< \left(\frac{x_0}{4\cdot 6^{2n_k}(2n_k)^{10n_k+1}\abs{\Delta_k}^{4n_k}}\right)^{1/2n_k} \]
and which is not totally complex in the case that $k$ is totally real, then $\Sys(\mathcal{C}(k,B)) \leq x_0$. 
\end{prop}

\begin{proof}
As in the proof of Proposition \ref{prop:silvermanapplication} we will prove Proposition \ref{prop:brindzaapplication} in the case that $k$ has a unique complex place and leave the totally real case to the reader. As every real place of $k$ ramifies in $B$, the Albert--Brauer--Hasse--Noether theorem implies that $L$ is totally complex. It now follows from Dirichlet's unit theorem that every system of fundamental units of $L/k$ contains a fundamental unit $u_0\in\mathcal O_L^*$ such that $u_0^n\not\in \mathcal O_k^*$ for any $n\geq 1$. Let $\sigma$ denote the nontrivial automorphism of $\Gal(L/k)$ and define $u=u_0/\sigma(u_0)$. It is clear that $\Norm_{L/k}(u)=1$ and that $u^n\not\in \mathcal O_k^*$ for any $n\geq 1$.

Theorem \ref{theorem:brindzabound} above, along with Lemmas 4.4, 4.5 of \cite{LMPT} and elementary properties of Weil heights, shows that we may assume that $u$ satisfies
\[ h(u)\leq 2\cdot 6^{2n_k}(2n_k)^{10n_k}\abs{\Delta_L}^{2n_k}\leq 2\cdot 6^{2n_k}(2n_k)^{10n_k}\abs{\Delta_k}^{4n_k}\abs{\Delta_{L/k}}^{2n_k}. \]
Let $\mathcal O$ be a maximal order of $B$ which contains the quadratic order $\mathcal O_k[u]$ and let $\gamma$ denote the image in $\Gamma_{\mathcal O}^1$ of $u$. The proposition now follows from \cite[Lemma 12.3.3]{MR} and the fact that the logarithm of the Mahler measure of the minimal polynomial of $\gamma$ is at most $2n_kh(u)$.
\end{proof}

%---------------------------------------------------------------------------------
\subsection{A mean value theorem and applications}

Let $k$ be a number field. A complex-valued function $f$ defined on the (nonzero) ideals of $\calO_k$ is \textbf{multiplicative} if $f(\mathfrak{a} \mathfrak{b}) = f(\mathfrak{a})f(\mathfrak{b})$ whenever $\mathfrak{a}$ and $\mathfrak{b}$ are coprime. When $k=\Q$, the following mean value theorem appears as Theorem 01 in \cite[Ch 0]{HT}, and the proof is sketched in Exercise 01 there. For details, see \cite[p. 58]{SchwarzSpilker}. The argument for general number fields $k$ can be carried out in precisely the same way, simply by replacing the sums over natural numbers with sums over integral ideals, and so we omit it.

\begin{prop}\label{prop:HT}
Let $k$ be a number field with a nonnegative-valued multiplicative function $f$ on the ideals of $\calO_k$. Suppose there are constants $A \ge 0$ and $B \ge0$ for which the following hold: For all $y\ge 0$,
\begin{equation} \sum_{|\mathfrak{p}| \le y} f(\pp) \log{|\pp|} \le Ay, \label{eq:Aconstant}\end{equation}
and
\begin{equation}\label{eq:Bconstant} \sum_{\mathfrak{p}} \sum_{\nu \ge 2} \frac{f(\mathfrak{p}^{\nu})}{|\mathfrak{p}^{\nu}|} \log |\mathfrak{p}^{\nu}| \le B. \end{equation}
(Here the sums on $\mathfrak{p}$ are over nonzero prime ideals of $\calO_k$.) Then for all $x> 1$,
\begin{align*} \sum_{|\mathfrak{a}| \le x} f(\mathfrak{a}) &\le (A+B+1) \frac{x}{\log{x}} \sum_{|\mathfrak{a}|\le x} \frac{f(\mathfrak{a})}{|\mathfrak{a}|}
\le (A+B+1) \frac{x}{\log{x}}  \prod_{|\mathfrak{p}| \le x}\left(1+\frac{f(\mathfrak{p})}{|\mathfrak{p}|} + \frac{f(\mathfrak{p}^2)}{|\mathfrak{p}^2|} + \dots\right).  \end{align*}
\end{prop}

Suppose that at prime power ideals, the values of $f$ are bounded by a constant depending at most on $k$. Then \eqref{eq:Aconstant}, \eqref{eq:Bconstant} hold with constants $A$ and $B$ depending only on $k$. Indeed, the bound $\sum_{|\pp| \le y} \log{|\pp|} \ll_{k} y$ is a crude consequence of Landau's prime ideal theorem. Moreover, since $\log|\pp^{\nu}| \ll |\pp|^{\nu/3}$,
\[ \sum_{\mathfrak{p}} \sum_{\nu \ge 2} \frac{\log|\pp^{\nu}|}{|\pp^{\nu}|} \ll  \sum_{\mathfrak{p}} \sum_{\nu \ge 2} |\pp|^{-2\nu/3} \ll \sum_{\pp} |\pp|^{-4/3}. \]
The final sum  on $\pp$ is bounded above by $\zeta_{k}(4/3)$. Thus, we can choose a constant $B$, depending on $k$, such that \eqref{eq:Bconstant} holds.

The following sieve lemma is a simple consequence of Proposition \ref{prop:HT}.

\begin{lem}\label{lem:sieve}
Let $k$ be a number field, and let $\Pp$ be a set of prime ideals of $\calO_k$. Suppose $g$ is a nonnegative-valued multiplicative function on the ideals of $\calO_k$, and that the values of $g$ on prime power ideals are $O_k(1)$. Then for $x\ge 2$,
\[ \sum_{\substack{|\mathfrak{a}| \le x \\ \mathfrak{a} \text{ squarefree}\\ \gcd(\mathfrak{a},\Pp)=1}} g(\mathfrak{a}) \ll_{k} x \prod_{\substack{|\pp|\le x \\ \pp \notin \Pp}}\bigg(1+\frac{g(\pp)-1}{|\pp|}\bigg) \prod_{\substack{|\pp|\le x\\\pp \in \Pp}} \bigg(1-\frac{1}{|\pp|}\bigg). \]
Here ``$\gcd(\mathfrak{a},\Pp)=1$'' denotes the condition that $\mathfrak{a}$ not be divisible by any member of $\Pp$.
\end{lem}

\begin{rmk}
The case when $g$ is identically $1$ is particularly important. In that case, Lemma \ref{lem:sieve} shows that the number of squarefree $\mathfrak{a}$ with $|\mathfrak{a}| \le x$ and $\gcd(\mathfrak{a},\Pp)=1$ is
\begin{equation}\label{eq:upperboundsieve} \ll_{k} x \prod_{\substack{|\pp|\le x\\ p\in \Pp}}\left(1-\frac{1}{|\pp|}\right). \end{equation}
\end{rmk}

\begin{proof}[Proof of Lemma \ref{lem:sieve}]
Let $f(\mathfrak{a}) = g(\mathfrak{a}) \mu(\mathfrak{a})^2 \1_{\gcd(\mathfrak{a},\Pp)=1}$, where $\1_{\gcd(\mathfrak{a}, \Pp)=1}$ denotes the charactistic function detecting those ideals for which $\gcd(\mathfrak{a}, \Pp) =1$. Then $f$ is multiplicative and its values at prime power ideals are $O_k(1)$. Applying Proposition \ref{prop:HT},
\[ \sum_{\substack{|\mathfrak{a}| \le x \\ \mathfrak{a} \text{ squarefree}\\ \gcd(\mathfrak{a},\Pp)=1}} g(\mathfrak{a}) = \sum_{|\mathfrak{a}| \le x} f(\mathfrak{a}) \ll_{k} \frac{x}{\log{x}} \prod_{\substack{|\pp|\le x \\ \pp \notin \Pp}}\bigg(1+\frac{g(\pp)}{|\pp|}\bigg).  \]
We now use the estimate $\frac{1}{\log{x}} \ll_{k} \prod_{|\pp| \le x}\bigg(1-\frac{1}{|\pp|}\bigg)$, which is a crude version of Mertens' theorem for number fields. (For a sharper, asymptotic version, see \cite{Rosen}.) Inserting this above gives
\begin{align*} \sum_{\substack{|\mathfrak{a}| \le x \\ \mathfrak{a} \text{ squarefree}\\ \gcd(\mathfrak{a},\Pp)=1}} g(\mathfrak{a}) &\ll_{k} x \Bigg(\prod_{\substack{|\pp| \le x \\ \pp \notin\Pp}}\bigg(1-\frac{1}{|\pp|}\bigg)\bigg(1+\frac{g(\pp)}{|\pp|}\bigg)\Bigg) \Bigg(\prod_{\substack{|\pp| \le x \\ \pp \in \Pp}}\left(1-\frac{1}{|\pp|}\right)\Bigg) \\
&\le x \prod_{\substack{|\pp|\le x \\ \pp \notin \Pp}}\bigg(1+\frac{g(\pp)-1}{|\pp|}\bigg) \prod_{\substack{|\pp|\le x\\\pp \in \Pp}} \bigg(1-\frac{1}{|\pp|}\bigg).\qedhere \end{align*}
\end{proof}

Define a multiplicative function $\Phi$ on the non-zero integral ideals of $k$ as follows: \[\Phi(\mathfrak a)=\abs{\mathfrak a}\prod_{\mathfrak p \mid \mathfrak a} \left(1-\frac{1}{\abs{\mathfrak p}}\right).\]

\begin{theorem}\label{theorem:countbyphi} Let $L/k$ be a quadratic extension of number fields and $x\ge 2$. The number of quaternion algebras $B$ over $k$ with $\Phi(\disc_f(B))\le x$ and which admit an embedding of $L$ is $\ll_{k,L} \frac{x}{(\log{x})^{1/2}}$.
\end{theorem}
\begin{proof}
Since $B$ is determined by its discriminant $\disc(B)$, it is enough to establish the stated upper bound for the number of possible values of $\disc(B)$. Since $k$ has finitely many infinite places, there are only $O_{k}(1)$ possibilities for the infinite component of $\disc(B)$. Thus, it will suffice to show that the given conditions restrict $\disc_{f}(B)$ to a set of size $O_{k,L}(x/(\log x)^{1/2})$.
Let $\Pp$ denote the set of prime ideals of $k$ that split in $L$.

In what follows, we use $\mathfrak{d}$ to denote a squarefree ideal of $\calO_k$ not divisible by any member of $\Pp$. Since $\disc_f(B)$ is such an ideal, it suffices to show the stated bound for the number of $\mathfrak{d}$ with $\Phi(\mathfrak{d}) \le x$.

We begin by estimating a second moment. Using Lemma \ref{lem:sieve}, we have for any $y\ge 3$ that
\begin{equation}\label{eq:2ndmoment} \sum_{|\mathfrak{d}| \le y} \left(\frac{|\mathfrak{d}|}{\Phi(\mathfrak{d})}\right)^2 \ll_{k} y\prod_{\substack{|\pp|\le y \\ \pp \notin \Pp}}\bigg(1+\frac{(|\pp|/\Phi(\pp))^2-1}{|\pp|}\bigg) \prod_{\substack{|\pp|\le y\\\pp \in \Pp}} \bigg(1-\frac{1}{|\pp|}\bigg). \end{equation}
Now
\[ \frac{(|\pp|/\Phi(\pp))^2-1}{|\pp|} = \frac{(2|\pp| - 1)}{(|\pp|-1)^2|\pp|} =  O\left(\frac{1}{|\pp|^2}\right). \]
Noting that $\displaystyle\sum_{\pp} \frac{1}{|\pp|^2} < \infty$ and that $1+t\le \exp(t)$ for all real $t$, we deduce that the first right-hand product in \eqref{eq:2ndmoment} satisfies
\[ \prod_{\substack{|\pp|\le y \\ \pp \notin \Pp}}\bigg(1+\frac{(|\pp|/\Phi(\pp))^2-1}{|\pp|}\bigg) \le \exp\left(\sum_{\pp} O\left(\frac{1}{|\pp|^2}\right)\right) \ll_{k} 1.\]
To handle the second product, we can use the Chebotarev density theorem, according to which the primes of $k$ that split in $L$ have density $\frac{1}{2}$. From a version of that theorem with a reasonable error term (e.g., the version of the theorem given in \cite{Schulze}), along with partial summation, we have
\[ \sum_{\substack{|\pp| \le y \\ \pp \in \Pp}}\frac{1}{|\pp|} = \frac{1}{2}\log\log|y| + O_{L,k}(1).\]
Since $\log\left(1 - \frac{1}{|\pp|}\right) = -\frac{1}{|\pp|} + O\left(\frac{1}{|\pp|^2}\right)$, it follows that the second right-hand product above is $O_{L,k}((\log{y})^{-1/2})$, and so collecting everything, $\sum_{|\mathfrak{d}| \le y} \left(\frac{|\mathfrak{d}|}{\Phi(\mathfrak{d})}\right)^2 \ll_{k,L} \frac{y}{(\log{y})^{1/2}}$.

We now return to counting $\mathfrak{d}$ with $\Phi(\mathfrak{d}) \le x$. Taking $y=x$ in the last estimate and noting that the summands are all at least $1$, we see there are only $O(x/(\log{x})^{1/2})$ possible $\mathfrak{d}$ with $|\mathfrak{d}| \le x$. Now let $y=2^\ell x$, where $\ell$ is a nonnegative integer. Observe that if $|\mathfrak{d}| > y$ but $\Phi(\mathfrak{d}) \leq x$, then $(|\mathfrak{d}|/\Phi(\mathfrak{d}))^2 > (y/x)^2 = 4^\ell$. Hence,
\begin{align*} \#\{\mathfrak{d}: |\mathfrak{d}| \in (y,2y], \Phi(\mathfrak{d}) \le x\} &\le \frac{1}{4^\ell} \sum_{|\mathfrak{d}| \le 2y} \left(\frac{|\mathfrak{d}|}{\Phi(\mathfrak{d})}\right)^2 \ll \frac{1}{4^\ell} \frac{2^{\ell+1}x}{\log(2^{\ell+1}x)^{1/2}} \ll \frac{1}{2^\ell} \frac{x}{(\log{x})^{1/2}}. \end{align*}
Summing on $\ell$, we find that the total number of $\mathfrak{d}$ with $|\mathfrak{d}| > x$ but $\Phi(\mathfrak{d}) \le x$ is also $O(x/(\log{x})^{1/2})$.
\end{proof}

%---------------------------------------------------------------------------------
\subsection{Proof of Theorem \ref{theorem:shortgeodesiccommensurabilityclasses}}

Let $x_0$ be a positive real number and $k$ be a number field which is totally real (respectively has a unique complex place). Recall that $N_k(V; x_0)$ is the number of commensurability classes $\mathcal{C}$ of arithmetic hyperbolic $2$--orbifolds (respectively $3$--orbifolds) with invariant trace field $k$, $V_\mathcal{C} \leq V$, and $\Sys(\mathcal{C}) \leq x_0$. To prove Theorem \ref{theorem:shortgeodesiccommensurabilityclasses}, we must show that for all sufficiently large $x_0$ that $N_k(V;x_0) \asymp V/\log(V)^{\frac{1}{2}}$, where the implied constants depend only upon $k$ and $x_0$.

\begin{proof}[Proof of Theorem \ref{theorem:shortgeodesiccommensurabilityclasses}]
We give a proof in the case of arithmetic hyperbolic $2$--orbifolds and leave the $3$--orbifold case to the reader. Borel \cite[\S 7.3]{borel-commensurability} has shown that the covolume of $\Gamma_{\calO}^1$ is
\begin{equation}\label{E:Borel2Vol}
\covol(\Gamma_{\calO}^1)=\frac{8\pi \abs{\Delta_k}^{\frac{3}{2}}\zeta_k(2)\Phi(\disc_f(B))}{(4\pi^2)^{n_k}},
\end{equation}
where $\Phi(\disc_f(B))=\prod_{\frakp\mid\disc_f(B)}\left({\abs{\frakp}}-1\right)$. It follows from Proposition \ref{prop:silvermanapplication} that $N_k(V;x_0)$ is at most the number of isomorphism classes of quaternion algebras $B$ over $k$ which satisfy $\Phi(\disc_f(B))\leq c_kV$ and admit an embedding of some quadratic extension $L/k$ with norm of relative discriminant $\abs{\Delta_{L/k}}<e^{2(n_k+x_0)}$. Here $c_k$ is a positive constant depending only on $k$ and which can easily be made explicit via (\ref{E:Borel2Vol}).
A theorem of Datskovsky--Wright \cite{DW88} shows that as $x\to\infty$ the number of quadratic extensions $L/k$ with $\abs{\Delta_{L/k}}<x$ is $\sim \frac{\kappa_k}{2^{r_2}\zeta_k(2)}x$, where $\kappa_k$ is the residue at $s=1$ of the Dedekind zeta function $\zeta_k(s)$ of $k$ and $r_2$ is the number of complex places of $k$.

Theorem \ref{theorem:countbyphi} shows that for a fixed quadratic extension $L$ of $k$, the number of quaternion algebras $B$ over $k$ with discriminant satisfying $\Phi(\disc_f(B))<x$ and which admit an embedding of $L$ is less than $\frac{\delta_Lx}{\log(x)^{\frac{1}{2}}}$ for some constant $\delta_L$ depending on $L$ and $k$. Let $L_1,\dots,L_r$ be the quadratic extensions of $k$ satisfying $\abs{\Delta_{L_i/k}}<e^{2(n_k+x_0)}$ and define $\delta:=\max_{i=1,\dots,r}\delta_{L_i}$.

The discussion above shows that for sufficiently large $x_0$,
\begin{align*}
N_k(V;x_0) &\ll_k e^{2x_0}\delta V/\log(V)^{\frac{1}{2}} \ll_{k,x_0} V/\log(V)^{\frac{1}{2}}.
\end{align*}

We now prove that for sufficiently large $x_0$, $N_k(V;x_0) \gg_{k,x_0} V/\log(V)^{\frac{1}{2}}$. Let $x_0$ be large enough that there exists a quadratic extension $L$ of $k$ which has signature $(2,n_k-1)$ and satisfies
\[ \abs{\Delta_{L/k}}< \left(\frac{x_0}{4\cdot 6^{2n_k}(2n_k)^{10n_k+1}\abs{\Delta_k}^{4n_k}}\right)^{1/2n_k}. \] A minor modification to the proof of Theorem 1.7 of \cite{LMPT} shows that the number of quaternion algebras $B$ over $k$ which are unramified at a unique real place of $k$, admit an embedding of $L$ and which satisfy $\abs{\disc_f(B)}<V$ is \begin{equation}\label{eq:delange0}\gg V/\log(V)^{\frac{1}{2}},\end{equation} as $V\to \infty$. We sketch the proof for the convenience of the reader. Counting $B$ is equivalent to counting possible values of $\disc(B)$, since $\disc(B)$ determines $B$ (up to isomorphism). Since $L$ has signature $(2,n_k-1)$, exactly one real place of $k$ splits in $L$. Let $v_1, \dots, v_{n_k-1}$ be a list of the non-split real places. Since $L$ is to embed into $B$, the condition that $B$ is unramified at a unique real place of $k$ forces $\disc_{\infty}(B)$ to be the product of the $v_i$. The requirements on $\disc_f(B)$ are that $\disc_f(B)$ be a squarefree ideal of $\calO_k$, that $\mu(\disc_f(B))=(-1)^{n_k-1}$ (so that $\disc(B)$ is the product of an even number of places), and that $\disc_f(B)$ not be divisible by any prime ideal of $\calO_k$ that splits in $L$ (so that $L$ embeds into $B$). Our count of $B$ is therefore given by the sum of the coefficients of $n^{-s}$, for $n < V$, in the Dirichlet series
\[ \sum_{\mathfrak{d}} \frac{\frac{1}{2}(\mu(\mathfrak{d})^2 + (-1)^{n_k-1} \mu(\mathfrak{d})) f(\mathfrak{d})}{|\mathfrak{d}|^s}, \]
where $f(\mathfrak{d})$ is the characteristic function of those $\mathfrak{d}$ not divisible by any prime ideal of $\calO_k$ that splits in $L$. This series has nonnegative coefficients and is easily seen to converge for $\Re(s)>1$. Moreover, one can show that our series is also analytic for $\Re(s)\ge 1$, except for a ``pole'' of order $1/2$ at $s=1$ with positive ``residue''. Here ``pole'' and ``residue'' are meant in the sense required for the application of Delange's Tauberian theorem \cite{Delange}. (Intuitively, this last claim is coming from the fact that asymptotically half the primes of $k$ split in $L$.) We refer to \cite{LMPT} for details. Equation \eqref{eq:delange0} is now an immediate consequence of Delange's theorem.

Proposition \ref{prop:brindzaapplication} shows that for each commensurability class $\mathscr C(k,B)$, we have $$\Sys(\mathcal{C}(k,B)) \leq x_0.$$ Observing that $\abs{\disc_f(B)}>\Phi(\disc_f(B))$ we see that each of these classes $\mathscr C(k,B)$ satisfies
\begin{align*}
V_{\mathscr C(k,B)} &= \covol(\Gamma_{\calO}^1) = \frac{8\pi \abs{\Delta_k}^{\frac{3}{2}}\zeta_k(2)\Phi(\disc_f(B))}{(4\pi^2)^{n_k}} \\
&< \frac{8\pi \abs{\Delta_k}^{\frac{3}{2}}\zeta_k(2)\abs{\disc_f(B)}}{(4\pi^2)^{n_k}} \leq c_k \abs{\disc_f(B)} \leq c_k V,
\end{align*}
where $c_k$ is a positive constant which depends only on $k$. The theorem follows.\end{proof}

\begin{rmk}
We note that in our applications of results from \cite{LMPT} that one must take into account that in \cite{LMPT} the discriminant of a quaternion algebra $B$ over $k$ is defined to be the formal product of all infinite primes of $k$ ramifying in $B$ with the \textbf{square} of the product of all finite primes of $k$ ramifying in $B$, whereas the present paper defines $\disc(B)$ to be the product of all primes of $k$ (finite or infinite) ramifying in $B$.
\end{rmk}

%---------------------------------------------------------------------------------
\subsection{Proof of Corollary \ref{T:SGC-Density}}

We are now ready to prove Corollary \ref{T:SGC-Density} using Theorem \ref{theorem:shortgeodesiccommensurabilityclasses}.

\begin{proof}[Proof of Corollary \ref{T:SGC-Density}]
In light of Theorem \ref{theorem:shortgeodesiccommensurabilityclasses} it suffices to show that as $V\to\infty$, the number $N_k(V)$ of commensurability classes $\mathscr C$ of arithmetic hyperbolic $2$--orbifolds ($3$--orbifolds) having invariant trace field $k$ and $V_{\mathscr C}<V$ satisfies
\begin{equation}\label{E:ClassCount}
N_k(V) \gg V,
\end{equation}
where the implied constant depends only on $k$. Indeed, we would have $\frac{N_k(V;x_0)}{N_k(V)} \ll \frac{1}{\log(V)^{1/2}}$, and hence verify the density zero claim. As above we will give a proof in the case of arithmetic hyperbolic $2$--orbifolds and leave the case of arithmetic hyperbolic $3$--orbifolds to the reader.

Let $k$ be a totally real field, $B$ a quaternion algebra over $k$ in which a unique real place splits and $\calO$ a maximal order of $B$. If for some $V>0$, we have $\abs{\disc_f(B)}<\frac{V(4\pi^2)^{n_k}}{8\pi \abs{\Delta_k}^{\frac{3}{2}}\zeta_k(2)}$, then it follows from (\ref{E:Borel2Vol}) that we must have
\begin{equation}\label{E:VolDisc}
V_{\mathscr C}:=\covol(\Gamma_{\calO}^1)<V,
\end{equation}
where $\mathscr C=\mathscr C(k,B)$.

A slight modification of \cite[Thm 1.5]{LMPT} shows that as $V\to\infty$, the number $N_{k,quat}(V)$ of quaternion algebras $B$ over $k$ which are unramified at a unique real place of $k$ and have $\abs{\disc_f(B)}<cV$ satisfies
\begin{equation}\label{E:QuatCount}
N_{k,quat}(V) \gg V,
\end{equation}
where the implied constant depends only on $k$. We sketch the proof of \eqref{E:QuatCount}. Fix a list of all of the real places of $k$, with one excluded (chosen arbitrarily), say $v_1, \dots, v_{n_k - 1}$. We will count $B$ which have $\disc_{\infty}(B)$ equal to the product of the $v_i$. The only requirements on $\disc_f(B)$ are that $\disc_f(B)$ be squarefree with $\mu(\disc_f(B)) = (-1)^{n_k-1}$ and $|\disc_f(B)| < cV$. It follows that our count of $B$ is bounded below by the sum of the coefficients of $n^{-s}$ with $n < cV$ in the Dirichlet series
\[ \sum_{\mathfrak{d}} \frac{\frac{1}{2}(\mu(\mathfrak{d})^2 + (-1)^{n_k-1} \mu(\mathfrak{d}))}{|\mathfrak{d}|^s}. \]
Clearly, the coefficients of this Dirichlet series are nonnegative. Moreover, it is straightforward to prove (by relating the series to that for $\zeta_k(s)$ and using known properties of the latter) that our series converges for $\Re(s) > 1$ and is analytic for $\Re(s) \ge 1$, except for a simple pole at $s=1$ with positive residue. Delange's theorem now implies that the sum of the coefficients of $n^{-s}$ for $n < X$ is asymptotic to a positive constant multiple of $X$, as $X\to\infty$. Taking $X=cV$, equation \eqref{E:QuatCount} follows.

From (\ref{E:VolDisc}) and (\ref{E:QuatCount}), we obtain (\ref{E:ClassCount}) as needed.
\end{proof}

Note that in the proof of Corollary \ref{T:SGC-Density} we showed that $N_k(V)\gg V$. It is also the case that $N_k(V)\ll V$. Indeed, minor modifications to the proof of Theorem \ref{theorem:countbyphi} show that the number of quaternion algebras $B$ over $k$ with $\Phi(\disc_f(B))<V$ is $\ll_k V$. The formula for the volume of $\mathscr C(k,B)$ now shows that $N_k(V)\ll V$, giving us the following.
\begin{cor}\label{cor:CommAsyRate}
For any totally real number field $k$ (respectively, number field with exactly one complex place), we have $N_k(V) \asymp V$, where the implied constant depends upon $k$.
\end{cor}

 %---------------------------------------------------------------------------------
%%% Applications to systoles
%---------------------------------------------------------------------------------
\section{Systolic growth of arithmetic hyperbolic surfaces}\label{Sec:Systole}

The main geometric goal of this section is the proof of Theorem \ref{T:SystoleInequality}.

%---------------------------------------------------------------------------------
%---------------------------------------------------------------------------------
\subsection{Counting quaternion algebras into which few quadratic fields embed}

We start with a counting result that we will use with Proposition \ref{prop:silvermanapplication} in our proof of Theorem \ref{T:SystoleInequality}.

\begin{theorem}\label{theorem:countingresult2}
Let $h(x)$ be any function which is $o(\log(x)^{\frac{1}{2}})$, $N_{\mathbf{Q},quat}(x)$ be the number of quaternion algebras $B$ over $\mathbf{Q}$ with $\disc_f(B)<x$ and $N'_{\mathbf{Q},quat}(x;h)$ be the number of quaternion algebras $B$ over $\mathbf{Q}$ with $\disc_f(B)<x$ and which do not admit an embedding of any quadratic field having absolute value of discriminant less than $h(x)$. Then as $x\to\infty$, $N'_{\mathbf{Q},quat}(x;h)\sim N_{\mathbf{Q},quat}(x)$.
\end{theorem}

To prove Theorem \ref{theorem:countingresult2}, we require the following lemma.

\begin{lem}\label{lemma:counting2lemma}
Let $K$ be a quadratic field. The number $N_{\mathbf{Q},quat}(x;K)$ of quaternion algebras $B$ over $\mathbf{Q}$ with $\disc_f(B) < x$ which admit an embedding of $K$ satisfies
\[ N_{\mathbf{Q},quat}(x;K)\ll \frac{x}{(\log{x})^{1/2}} \bigg(\frac{\abs{\Delta_K}}{\phi(\abs{\Delta_K})}\bigg)^{1/2} \prod_{p \le x}\bigg(1-\frac{\leg{\Delta_K}{p}}{p}\bigg)^{1/2}. \]
Here the implied constant is absolute.
\end{lem}

\begin{proof}
Since $B$ is an algebra over $\Q$, the squarefree integer $\disc_f(B)$ determines $B$. Moreover, if $K$ embeds into $B$, then $\disc_f(B)$ is not divisible by any prime from $\Pp$, where $\Pp$ is the set of rational primes that split in $K$.
If $p$ is a rational prime not dividing $\Delta_K$ (equivalently, unramified in $K$), then $\leg{\Delta_K}{p}=1$ or $-1$, according to whether $p$ is split or inert in $K$, respectively; hence,
\[ \frac{1}{2}\left(1 + \leg{\Delta_K}{p}\right) =
\begin{cases}
1 &\text{if $p \in \Pp$},\\
0 &\text{otherwise}.
\end{cases}
\]
Since $\disc_f(B) < x$, equation \eqref{eq:upperboundsieve} (with $k=\Q$) shows that the number of possibilities for $B$ is
 \begin{equation}\label{eq:numberofB}\ll x \prod_{\substack{p \le x \\ p \nmid \Delta_K}}\bigg(1-\frac{\left(1+\leg{\Delta_K}{p}\right)/2}{p}\bigg) \le x \prod_{\substack{p \le x}}\bigg(1-\frac{\left(1+\leg{\Delta_K}{p}\right)/2}{p}\bigg) \prod_{p \mid \Delta_K}\bigg(1-\frac{1}{2p}\bigg)^{-1}. \end{equation}
Using that $\log(1+t) = t + O(t^2)$ for $|t| \le \frac12$,
\[ \log \bigg(1-\frac{\left(1+\leg{\Delta_K}{p}\right)/2}{p}\bigg) = \log \left(\bigg(1-\frac{1}{p}\bigg)^{1/2} \bigg(1-\frac{\leg{\Delta_K}{p}}{p}\bigg)^{1/2}\right) + O\left(\frac{1}{p^2}\right),\]
and similarly
\[ \log\left(\left(1-\frac{1}{2p}\right)^{-1}\right) = \log\left(\left(1-\frac{1}{p}\right)^{-1/2}\right) + O\left(\frac{1}{p^2}\right).\]
Now exponentiate. Keeping in mind that
\[ \prod_{p \le x}\left(1-\frac{1}{p}\right) \ll \frac{1}{\log{x}}, \quad \prod_{p \mid \Delta_K}\left(1-\frac{1}{p}\right)^{-1} = \frac{|\Delta_K|}{\phi(|\Delta_K|)}, \quad \sum_{p}\frac{1}{p^2} < \infty, \]
we see that the final expression in \eqref{eq:numberofB} is
\[ \ll \frac{x}{(\log{x})^{1/2}} \bigg(\frac{\abs{\Delta_K}}{\phi(\abs{\Delta_K})}\bigg)^{1/2} \prod_{p \le x}\bigg(1-\frac{\leg{\Delta_K}{p}}{p}\bigg)^{1/2}, \]
where the implied constants are absolute.
\end{proof}

We now prove Theorem \ref{theorem:countingresult2}.

\begin{proof}[Proof of Theorem \ref{theorem:countingresult2}]
Let $H$ be a parameter assumed to satisfy $H \le \log{x}$. Let $K$ be a quadratic field with $\abs{\Delta_K} \le H$, and let $\chi(\cdot) = \leg{\Delta_K}{\cdot}$ be the associated quadratic Dirichlet character. Let us estimate the Euler product factor appearing in the upper bound of Lemma \ref{lemma:counting2lemma}. Since
\[ \log\left(1-\frac{\chi(p)}{p}\right) = -\frac{\chi(p)}{p} + O\left(\frac{1}{p^2}\right) \]
and $\sum_{p} \frac{1}{p^2} < \infty$, we see that
\begin{equation}\label{E:L-Value}
L(1,\chi)\prod_{p \le x}\bigg(1-\frac{\chi(p)}{p}\bigg) = \prod_{p > x}\bigg(1-\frac{\chi(p)}{p}\bigg) \ll \exp\left(-\sum_{p > x}\frac{\chi(p)}{p}\right).
\end{equation}
Since $\chi$ is a primitive character of conductor $\abs{\Delta_K}$, and $\abs{\Delta_K} \le H \le \log{x}$, the prime number theorem for progressions implies that $\sum_{p \le T} \chi(p) \ll T/(\log{T})^2$, uniformly for $T \ge x$. (Here we use \cite[eq. (8), p.123]{davenport}, along with the bound $\beta_1 < 1-c/q^{1/2}\log^2{q}$ coming from Dirichlet's class number formula.) Hence, by partial summation, $\sum_{p > x}\frac{\chi(p)}{p}=O(1)$. With (\ref{E:L-Value}), the above yields
\[ \prod_{p \le x}\bigg(1-\frac{\chi(p)}{p}\bigg) \ll L(1,\chi)^{-1}. \]

It now follows from Lemma \ref{lemma:counting2lemma} that the number $N_{\Q,quat}(H;K)$ of quaternion algebras $B/\Q$ that admit an embedding of some quadratic field $K$ with $\abs{\Delta_K} \le H$ satisfies
\[ N_{\Q,quat}(H;K) \ll \frac{x}{(\log{x})^{1/2}} \sum_{\abs{\Delta_K} \le H} \bigg(\frac{|\Delta_K|}{\phi(\abs{\Delta_K})}\bigg)^{1/2} L(1,\leg{\Delta_K}{\cdot})^{-1/2}. \]
By Cauchy--Schwarz, the sum on $\Delta_K$ is
\[ \ll \bigg(\sum_{\abs{\Delta_K} \le H} \frac{\abs{\Delta_K}}{\phi(\abs{\Delta_K})}\bigg)^{1/2}\bigg(\sum_{\abs{\Delta_K} \le H} L(1,\leg{\Delta_K}{\cdot})^{-1}\bigg)^{1/2}.   \]
The first sum is $O(H)$. In fact, it is well-known that the arithmetic function $n/\phi(n)$ has finite moments of every order (see, e.g., \cite[Ex 14, p. 42]{MV}). From \cite[Thm 2]{GS} (with $z=-1$), and the subsequent comment there about Siegel's theorem, the second sum on $\Delta_K$ is also $O(H)$. We conclude that $N_{\Q,quat}(H;K) \ll \frac{x H}{(\log{x})^{1/2}}$. Finally, let $H = h(x)$. Since $h(x)=o((\log{x})^{1/2})$, our upper bound is $o(x)$. Since $N_{\Q,quat}(x) \gg x$ from \cite[Thm 1.5]{LMPT}, the theorem follows.
\end{proof}

\begin{rmk}
The above argument is similar to the proof of \cite[Lemma 2.6]{FP}.
\end{rmk}

%---------------------------------------------------------------------------------
%---------------------------------------------------------------------------------
\subsection{Applications to the systole growth of arithmetic hyperbolic surfaces}

Recall that for a hyperbolic $2$--orbifold $M$, we denote the systole of $M$ by $\Sys(M)$, which is the length of the shortest closed geodesic on $M$.

\begin{theorem}
Let $S^1$ be the set of all arithmetic Fuchsian groups of the form $\Gamma_{\mathcal O}^1$ where $\mathcal O$ is a maximal order of an indefinite quaternion division algebra over $\mathbf{Q}$ and $S^{min}$ be the set of all maximal arithmetic Fuchsian groups with invariant trace field $\mathbf{Q}$ which have minimal covolume within their commensurability class. Then for all $\epsilon>0$ the following are true:

\begin{enumerate}\label{systole}
\item The set of $\Gamma\in S^{1}$ such that \[\Sys(\mathbf{H}^2/\Gamma)>(\frac{1}{4}-\epsilon)\log\log\left(\frac{3}{\pi}\Vol(\mathbf{H}^2/\Gamma)\right)\] has density one in $S^{1}$.
\item The set of $\Gamma\in S^{min}$ such that \[\Sys(\mathbf{H}^2/\Gamma)>(\frac{1}{4}-\epsilon)\log\log\left(\frac{24}{\pi}\Vol(\mathbf{H}^2/\Gamma)\right)\] has density one in $S^{min}$.
\end{enumerate}
\end{theorem}

\begin{proof}
Let $h(x)=\log(x)^{\frac{1}{2}-\epsilon}$ and $B$ be an indefinite quaternion division algebra over $\textbf{Q}$ which does not admit an embedding of any quadratic field with absolute value of discriminant less than $h(\disc_f(B))$. Note that by Theorem \ref{theorem:countingresult2} and its proof, the set of such quaternion algebras has density one within the set of all indefinite quaternion division algebras over $\textbf{Q}$. Let $\mathcal O$ be a maximal order in $B$ so that $\Gamma_{\mathcal O}^1\in\mathscr C(\textbf{Q},B)$ and $\Gamma^{min}$ be the element of $\mathscr C(\textbf{Q},B)$ with minimal covolume. We remark that because every indefinite quaternion algebra defined over $\textbf{Q}$ has type number $1$, there is a one-to-one correspondence between arithmetic Fuchsian groups of the form $\Gamma_{\mathcal O}^1$, arithmetic Fuchsian groups of the form $\Gamma^{min}$ and commensurability classes $\mathscr C(\textbf{Q},B)$. The first assertion now follows from Proposition \ref{prop:silvermanapplication}, Theorem \ref{theorem:countingresult2} and equation (\ref{volC}). The second assertion follows from the same reasoning along with the fact that (see \cite{borel-commensurability} and \cite[Lemma 4.1]{chinburg-smallestorbifold}) $[\Gamma^{min}:\Gamma_{\mathcal O}^1]=2^{1+\#\Ram_f(B)}$.
\end{proof}

As a consequence of Theorem \ref{systole}, we obtain Theorem \ref{T:SystoleInequality}. The details are as follows.

\begin{proof}[Proof of Theorem \ref{T:SystoleInequality}]
In light of Theorem \ref{systole} it suffices to show that for all $\Gamma\in\mathscr C(\textbf{Q},B)$, where $\mathscr C(\textbf{Q},B)$ lies within a set of commensurability classes of density one, that $$\Sys(\textbf{H}^2/\Gamma)>\frac{1}{2}\Sys(\textbf{H}^2/\Gamma^{min}).$$ To that end, let $\gamma\in\Gamma$ be a hyperbolic element whose associated geodesic has length $\ell(\gamma)$. The subgroup $\Gamma^{(2)}$ generated by squares of elements in $\Gamma$ is contained in $\Gamma_{\mathcal O}^1$ for some maximal order $\mathcal O$ of $B$ (see \cite[Ch 8]{MR}). As $\Gamma^{min}$ is derived from the normalizer $N(\mathcal O)$ of $\mathcal O$ in $B^*$, we see that up to isomorphism $\gamma^2\in \Gamma^{(2)}\subset \Gamma_{\mathcal O}^1\subset \Gamma^{min}$. As $\ell(\gamma^2)=2\ell(\gamma)$, the corollary follows from Theorem \ref{systole}.
\end{proof}

%---------------------------------------------------------------------------------
%%%Arithmetic manifolds with totally geodesic surfaces of small area
%---------------------------------------------------------------------------------
\section{Theorem \ref{T:SSC-Density}: Totally geodesic surfaces of small area}\label{Sec:SSC}

Theorem \ref{T:SSC-Density} is an immediate consequence of the following theorem.

\begin{theorem}\label{theorem:surfaces}
Let $k$ be a totally real field and $\mathscr C(k, B_0)$ be a commensurability class of arithmetic hyperbolic surfaces with invariant trace field $k$. The set of commensurability classes of arithmetic hyperbolic $3$--orbifolds with a representative containing a totally geodesic $2$--orbifold in $\mathscr C(k, B_0)$ has density zero within the set of all commensurability classes of arithmetic hyperbolic $3$--orbifolds with invariant trace field a quadratic extension of $k$.
\end{theorem}

\begin{proof}
Recall that the volume of a commensurability class $\mathscr C=\mathscr C(L,B)$ of arithmetic hyperbolic $3$--orbifolds with invariant trace field $L$ is

\begin{equation}\label{volC}
V_\mathscr{C}=\covol(\Gamma_{\calO}^1)=\frac{\abs{\Delta_L}^{\frac{3}{2}}\zeta_L(2)\Phi(\disc_f(B))}{(4\pi^2)^{n_k-1}},
\end{equation}
where $\mathcal O$ is a maximal order of $B$. The results of \cite[Ch 9.5]{MR} show that a representative of $\mathscr C=\mathscr C(L,B)$ contains a primitive totally geodesic surface in $\mathscr C(k,B_0)$
if and only if $L$ is a quadratic field extension of $k$ with a unique complex place and $B_0\otimes_{k} L\cong B$. The theorem of Datskovsky--Wright \cite{DW88} shows that the number of such $\mathscr C(L,B)$ with $V_{\mathscr C(L,B)}<V$ is $\ll V^{\frac{2}{3}}$ for large $V$, where the implied constant depends on $\mathscr C(k,B_0)$. Suppose now that $L$ is a fixed quadratic extension of $k$ which has a unique complex place. Then there exists a constant $\delta_L>0$ such that the number of quaternion algebras $B$ over $L$ such that $V_{\mathscr C(L,B)}<V$ is $\gg \delta_L V$ for large $V$ (see \cite[Thm 1.5]{LMPT} and the remark which follows the proof of Theorem \ref{theorem:shortgeodesiccommensurabilityclasses}), which already proves the theorem.
\end{proof}

The following is an immediate consequence of Theorem \ref{theorem:surfaces}, the fact that there are only finitely many arithmetic hyperbolic $2$--orbifolds of bounded volume \cite{borel-commensurability} and the fact that all totally geodesic $2$--orbifolds of an arithmetic hyperbolic $3$--orbifold must have the same invariant trace field (which must be the maximal totally real subfield of the invariant trace field of the $3$--orbifold).

\begin{cor}
For all $V>0$ the set of commensurability classes of arithmetic hyperbolic $3$--orbifolds with a representative containing a totally geodesic $2$--orbifold with area less than $V$ has density zero within the set of all commensurability classes of arithmetic hyperbolic $3$--orbifolds.
\end{cor}

\begin{rmk}
For a commensurability class $\mathcal{C}(L,B)$ of arithmetic hyperbolic 3--orbifolds, we can consider the minimal area of a totally geodesic, immersed 2--orbifold in $M$ as we vary $M \in \mathcal{C}(L,B)$. As in the case of systoles, there are infinitely many $3$--orbifolds in $\mathcal{C}(L,B)$ that contain a totally geodesic, immersed $2$--orbifold of minimal area. Unlike Lemma \ref{L:OrbMan1Sys}, these $3$--orbifolds might never be manifolds. This can fail for the trivial reason that the minimal area totally geodesic, immersed $2$--orbifolds in $M \in \mathcal{C}(L,B)$ need not be manifolds. If there is a minimal area totally geodesic $2$--manifold, then there exists a $3$--manifold $M \in \mathcal{C}(L,B)$ that contains this $2$--manifold as a totally geodesic submanifold. This follows from the non-trivial fact that the fundamental groups of complete, finite volume, hyperbolic $3$--orbifolds are subgroup separable \cite[Cor 9.4]{Agol}; see \cite[\S 7.3]{McR} for a discussion of how this follows from subgroup separability.
\end{rmk}

\section*{Acknowledgements} The authors would like to thank Ian Agol, Mikhail Belolipetsky, Bobby Grizzard, and Alan Reid for useful conversations on the material of this paper. The authors also thank the anonymous referees for comments that helped improve the mathematics and readability of this article. The first author was partially supported by NSF RTG grant DMS-1045119 and by an NSF Mathematical Sciences Postdoctoral Fellowship. The second author was partially supported by the NSF grants DMS-1105710 and DMS-1408458. The third author was partially supported by the NSF grant DMS-1402268. The fourth author was partially supported by a Max Planck Institute fellowship and by an AMS Simons Travel Grant.

%---------------------------------------------------------------------------------
%\bibliographystyle{ams}
%\bibliography{shortgeodesics_MRL_final.bib}

%---------------------------------------------------------------------------------
%---------------------------------------------------------------------------------

\end{document}